\documentclass[10pt]{amsart}

\usepackage{amsmath, amssymb,amsmath, graphicx, xypic}
\usepackage{hyperref}
\usepackage[small]{caption}
\usepackage{txfonts}
\usepackage{enumitem}

\numberwithin{equation}{section}

\DeclareMathOperator{\FV}{FV}
\DeclareMathOperator{\catm1}{{CAT}(-1)}
\DeclareMathOperator{\cat0}{{CAT}(0)}
\newcommand{\nclose}[1]{\ensuremath{\langle\!\langle#1\rangle\!\rangle}}

\def\N{\mathbb{N}} 
\def\Z{\mathbb{Z}}
\def\ZH{\mathbb{Z}H}
\def\ZG{\mathbb{Z}G}
\def\mc{\mathcal}
\def\cd{\mathsf{cd}}
\def\cdF{\mathsf{cd}_{\mc F}}
\def\gdF{\mathsf{gd}_{\mc F}}

\newcommand {\Or}{\mathcal{O}_{\mathcal F}(G)}
\newcommand {\OrH}{\mathcal{O}_{\mathcal F\cap H}(H)}

\def\ZGF{\Z_{G,\mc F}}

\theoremstyle{theorem}
\newtheorem{theorem}{Theorem}[section]
\newtheorem{lemma}[theorem]{Lemma}
\newtheorem{proposition}[theorem]{Proposition}
\newtheorem{corollary}[theorem]{Corollary}
\newtheorem{step}{Step}

\theoremstyle{remark}
\newtheorem{definition}[theorem]{Definition}

\newtheorem{remark}[theorem]{Remark}
\newtheorem{question}[theorem]{Question}

\DeclareMathOperator{\aut}{\mathsf{Aut}}

\DeclareMathOperator{\Dist}{\mathsf{Dist}}

\DeclareMathOperator{\kernel}{\mathsf{Ker}}

\begin{document}

\title[]{Subgroups of Relatively Hyperbolic Groups of Bredon Cohomological Dimension $2$}
\author[E.~Mart\'inez-Pedroza]{Eduardo Mart\'inez-Pedroza}
   \address{Memorial University\\ St. John's, Newfoundland and Labrador, Canada}
  \email{emartinezped@mun.ca}

\maketitle

\begin{abstract}
A remarkable result of Gersten states that the class of hyperbolic groups of cohomological dimension   $2$ is closed under taking finitely presented (or more generally $FP_2$) subgroups. We prove the analogous result for relatively hyperbolic groups of Bredon cohomological dimension $2$ with respect to the family of parabolic subgroups. A class of groups where our result applies consists of  $C'(1/6)$ small cancellation products.  The proof relies on an algebraic approach to relative homological Dehn functions, and a characterization of relative hyperbolicity in the framework of finiteness properties over Bredon modules and homological Isoperimetric inequalities. 
\end{abstract}


\section{Introduction}

It is a remarkable result of Gersten that the class of hyperbolic groups  of cohomological dimension at most $2$   is closed under taking finitely presented subgroups, see Theorem~\ref{thm:GerstenSubgp}.  The dimension assumption  is sharp since Brady exhibited a hyperbolic group of cohomological dimension $3$ containing a finitely presented subgroup which is not hyperbolic~\cite{Br99}.

\begin{theorem}[Gersten]\cite[Theorem 5.4]{Ge96}\label{thm:GerstenSubgp}
Let $G$ be a hyperbolic group and suppose $\mathsf{cd}(G)\leq 2$. 
If $H$ is a subgroup of type $FP_2$, then $H$ is hyperbolic.
\end{theorem}

In this article we show that such subgroup phenomena hold  in the class of relatively hyperbolic groups of Bredon cohomological dimension $2$ with respect to parabolic subgroups.    We use the approach to relative hyperbolicity by Bowditch~\cite{Bo12} which is equivalent to the approach by Osin~\cite{Os06} for finitely generated groups~\cite[Theorem 6.10]{Os06}. 

\begin{definition}(Bowditch relatively hyperbolic groups)\label{def:BowRH}\cite{Bo12}
A graph $\Gamma$ is \emph{fine} if every edge is contained in only finitely many circuits (embedded closed paths) of length $n$ for any integer $n$. A finitely generated  group $G$ is \emph{hyperbolic relative to a finite collection of subgroups $\mc P$} if
$G$ acts on a connected, fine, hyperbolic graph $\Gamma$ with finite edge stabilizers, finitely many orbits of edges,  and $\mc P$ is a set of representatives of distinct conjugacy classes of vertex stabilizers (such that each infinite stabilizer is represented).   Subgroups of $G$ which are conjugate to a subgroup of some $P\in \mc P$ are called \emph{parabolic subgroups}. 
\end{definition}

A \emph{family of subgroups} $\mc F$ of a group $G$ is a non-empty collection of subgroups closed under conjugation and subgroups. The theory of (right) modules over the orbit category $\mc{O_F}(G)$ was established by Bredon~\cite{Br67}, tom Dieck~\cite{tD87}  and  L{\"u}ck~\cite{Lu89}. In the case that $\mc F$ is the trivial family, the $\Or$-modules are $\ZG$-modules. In this context, the notions of cohomological dimension $\cdF(G)$ and finiteness properties $FP_{n, \mc F}$ for the pair $(G, \mc F)$ are defined analogously as their counterparts $\cd(G)$ and $FP_n$, see Section~\ref{sec:finiteness}. 

To simplify the notation in our statements we introduce the following non-standard terminology. A group $G$ is \emph{hyperbolic relative to a family of subgroups $\mc F$} if there is a finite collection of subgroups $\mc P$ generating the family $\mc F$, such that $G$ is hyperbolic relative to $\mc P$.  If $H\leq G$ and $\mc F$ is a family of subgroups, we denote by $\mc F \cap H$ the family of subgroups $\{K\cap H\colon K\in \mc F\}$.  Our main result is the following.

\begin{theorem} \label{thm:subgp}\label{thmi:main}
Let $G$ be hyperbolic relative to a family $\mc F$ containing all finite subgroups. If $\mathsf{cd}_{\mc F}(G)\leq 2$ and $H\leq G$ is of type $FP_{2, \mc F\cap H}$,
then $H$ is hyperbolic relative to $\mc F\cap H$.
\end{theorem}

Given a group $G$ and a family of subgroups $\mc F$, \emph{a model for $E_{\mc F}G$} is a $G$-complex $X$ such that stabilizers of all the isotropy groups of $X$ belong to $\mc F$, and for every $Q\in \mc F$ the fixed point set $X^Q$ is a non-empty contractible subcomplex of $X$. The geometric dimension $\gdF(G)$ is defined as the smallest dimension of a model for $E_{\mc F}G$. It is known that  \[ \cdF(G) \leq \gdF(G) \leq \max\{3, \cdF(G)\},\] 
see the work of L\"uck and Meintrup~\cite[Theorem 0.1]{LuMe00}. 

Examples of Brady, Leary and Nucinkis~\cite{BrLe01}, and Leary and Fluch~\cite{FlLe12} show that the equality of $\cdF(G)$ and $\gdF(G)$ fails for some non-trivial families. For the case that $\mc F$ is the trivial family, whether this equality holds is the outstanding Eilenberg-Ganea conjecture.   The use of  cohomological dimension instead of geometric dimension in Theorem~\ref{thmi:main} is justified by the unknown Eilenberg-Ganea phenomena for general families.

In practice, we expect most applications of Theorem~\ref{thmi:main} to be in the geometric case $\gdF(G)\leq 2$ which lies outside the abstract framework of Bredon modules.

\subsection{A version of Theorem~\ref{thmi:main} outside the framework of Bredon modules}

By a \emph{toral relatively hyperbolic group} we mean a finitely generated torsion-free group $G$ which is hyperbolic relative to a finite collection $\mc P$ of abelian subgroups. By~\cite[Theorem 1.1]{Os06}, all subgroups in $\mc P$ are finitely generated.

\begin{theorem}\label{cori:main}
Let $G$ be toral relatively hyperbolic group, let $\mc F$ be the family of parabolic subgroups, and suppose  $\gdF(G)\leq 2$. If $H\leq G$ is  finitely presented with finitely many conjugacy classes of maximal parabolic subgroups, then $H$ is toral relatively hyperbolic. 
\end{theorem}
    
Theorem ~\ref{cori:main} follows from  $\cdF(G) \leq \gdF(G)$ and Proposition~\ref{propi:fp2sufficient} below. Recall that a collection of subgroups $\mc P$ of a group $G$ is  \emph{ malnormal (almost malnormal)}  if for any $P,P'\in \mc P$ and $g\in G$, either $g^{-1}Pg\cap P'$ is  trivial (respectively, finite), or $P=P'$ and $g\in P$.  

\begin{proposition}[Proposition~\ref{prop:F2Fchar}]\label{propi:fp2sufficient} 
If $G$ is finitely presented and $\mc F$ is a family  generated by a  malnormal finite collection of finitely generated subgroups, then $G$ is $FP_{2, \mc F}$.
\end{proposition}

\begin{proof}[Proof of Theorem~\ref{cori:main}]
Since $G$ is torsion-free, any two maximal parabolic subgroups of $G$ have either trivial intersection or they are equal~\cite[Lemma 2.2]{MaWi10}. Hence the same statement holds for the maximal parabolic subgroups of $H$. Since all parabolic subgroups of $H$ are finitely generated,  Proposition~\ref{propi:fp2sufficient} implies that $H$ is $FP_{2, \mc F\cap H}$. Then Theorem~\ref{thmi:main} implies that $H$ is toral relatively hyperbolic.
\end{proof}

We raise the following question on the hypothesis of Theorem~\ref{cori:main}, 
 remark on an simple example that might shed some light on the situation, 
and quote a related result.

\begin{question}\label{qu:conjugacymax}
Let $G$ be a toral relatively hyperbolic group,  let $\mc F$ be the family of parabolic subgroups, and suppose  $\cdF(G)\leq 2$.  Suppose that $H$ is finitely presented. Is $H$  toral relatively hyperbolic? 
\end{question}

\begin{remark}\label{ex:Osin} There is a toral relatively hyperbolic $G$ with  $\cdF(G)\leq 2$, where $\mc F$ is the family of parabolic subgroups, and such that $G$ contains a free group of finite rank with infinitely many conjugacy classes of maximal parabolic subgroups.  In particular, $H$ is not hyperbolic relative to $\mc F\cap H$, the group $H$ is not $FP_{2, \mc F\cap H}$ as a consequence of~\cite[Proposition 3.6.1]{Jg11}, but $H$ is hyperbolic and hence toral relatively hyperbolic. 

Indeed, let $H$ be a free group of rank $2$, and let $G$ be hyperbolic free by cyclic group $H \rtimes_\varphi \Z$; for the existence of an automorphism $\varphi\in \aut(H)$ such that the resulting semidirect product is hyperbolic we refer the reader to~\cite{BeFe92}. Let $g$ be a non-trivial element of $H$ that is not a proper power, and let $\mc F$ be the family generated by $P=\langle g \rangle$. Then $G$ is hyperbolic relative to $\mc F$ by \cite[Theorem 7.11]{Bo12}. Since $H$ is normal in $G$, the subgroup $H$ has non-trivial intersection with any conjugate of $P$. Hence $H$ has infinitely many conjugacy classes of maximal parabolic subgroups, and in particular is not hyperbolic relative to $\mc F\cap H$.  It is left to remark that $\gdF(G)\leq 2$. Since $H$ has rank two, $G$ is the fundamental group of a finite $2$-dimensional locally $\cat0$ complex $X$, see~\cite{TBr95}.  Let $\gamma$ be a closed geodesic in $X$ representing the conjugacy class of $g$, and let $Y$ be the space obtained by taking the universal cover $\widetilde X$ of $X$ and coning-off each geodesic of $\widetilde X$ covering $\gamma$. Observe that $G$ acts on $Y$ and that $Y$ is a $2$-dimensional model for $E_{\mc F}G$.
\end{remark}

A $\catm1$ version of Theorem~\ref{cori:main}  was known  by Hanlon and the author~\cite[Theorem 5.15]{HaMa14}. This result addresses Question~\ref{qu:conjugacymax} in a particular geometric setting.  We quote this result in the framework of toral relatively hyperbolic groups, and remark that the statement in the cited reference is slightly more general.

\begin{theorem}\cite[Theorem 5.15]{HaMa14}\label{thmi:towers}
Let $G$ be a toral relatively hyperbolic group, let $\mc F$ be the family of parabolic subgroups, and suppose there is a cocompact $2$-dimensional  $\catm1$ model of $E_{\mc F}G$  with trivial edge stabilizers. If $H\leq G$ is  finitely presented, then $H$ is toral relatively hyperbolic. 
\end{theorem}

\subsection{An application to small cancellation products}
The studies on small cancellation groups by Greendlinger imply that a torsion-free $C'(1/6)$ group has cohomological dimension $2$, see~\cite{ChCoHu81}. An analogous statement holds for all small  cancellation products as defined in~\cite[Chapter V]{LySc01}.

\begin{theorem}[Theorem~\ref{prop:sm}]\label{propi:smallcancelationgd}
Let $G$ be the quotient of a  free product $\ast_{i=1}^m P_i / \nclose{\mc R}$ where $\mc R$ is a finite symmetrized set.   If $\mc R$ is satisfies the $C'(1/6)$-small cancellation hypothesis and $G$ is torsion-free, then $\gdF(G)\leq 2$ where $\mc F$ is the family generated by $\{P_i\colon 1\leq  i\leq m \}$.
\end{theorem}

A group $G$ as in the proposition is hyperbolic relative to  $\{P_i\}_{i=1}^m$, see~\cite[Page 4, (II)]{Os06}. Theorems~\ref{thmi:main} and~\ref{propi:smallcancelationgd} yield Corollary~\ref{cor:c16} below. 
The statement of the corollary for finitely presented $C'(1/6)$-small cancellation groups was known~\cite[Theorem 7.6]{Ge96}.

\begin{corollary}\label{cor:c16}
Let $G$ be a torsion-free quotient $\ast_{i=1}^m P_i / \nclose{\mc R}$ of a  free product $\ast_{i=1}^m P_i$ where $\mc R$ is a finite set satisfying the $C'(1/6)$-small cancellation condition. Let $\mc F$ be the family generated by $\{P_i\}_{i=1}^m$.
 If $H\leq G$ is  $FP_{2, \mc F\cap H}$, then $H$ is hyperbolic relative to $\mc F\cap H$. 
\end{corollary}

\subsection{A characterization of relative hyperbolicity} The proof of Theorem~\ref{thmi:main}, relies on a characterization of relatively hyperbolic groups stated below.  

 The homological Dehn function of a space is a generalized isoperimetric function describing the minimal volume required to fill  cellular $1$--cycles  with cellular $2$--chains, see Section~\ref{sec:InvarianceFVG} for definitions.   The homological Dehn function $\FV_G\colon \N \to \N$ of a finitely presented group $G$ is  the homological Dehn function of the universal cover of a $K(G,1)$ with finite $2$-skeleton;  the growth rate of this function is an invariant of the group~\cite{ABDY, Ge99, Ge96}.  Recall that a pair of functions $f$ and $g$ from $\mathbb N$ to $\mathbb N \cup \{ \infty \}$ have \emph{equivalent growth rate} if $f \preceq g$ and $g \preceq f$, where
$f \preceq g$ means that there exists $0 < C < \infty$ such that  $f(n) \leq C g( Cn + C) + Cn + C$  for all $n \in \mathbb N$.

Gersten observed that one can define $\FV_G$ only assuming that $G$ is an $FP_2$ group, and discovered the following characterization of word-hyperbolic groups.  

\begin{theorem}[Gersten]\label{thm:Gersten}\cite[Theorem 5.2]{Ge96}
A group $G$ is hyperbolic if and only if $G$ is $FP_2$ and $\FV_G(n)$ has linear growth. 
\end{theorem}

Given a group $G$ and a family of subgroups $\mc F$, under the assumption that $G$ is $FP_{2,\mc F}$, we define the \emph{homological Dehn function of $G$ relative to $\mc F$} which we denote by $\FV_{G,\mc F}$.  The growth rate of $\FV_{G,\mc F}$ is an invariant of the pair $(G, \mc F)$, see Theorem~\ref{AlgebraicDef}. In the case that $\mc F$ consists of only the trivial subgroup and $G$ is finitely presented (or more generally $FP_2$), the definition of $\FV_{G, \mc F}$ coincides with the homological Dehn function $\FV_G$.  
   
\begin{theorem}\label{thmi:rhypchar0}\label{thm:rhypchar0}
Let $G$ be a group and let $\mc P$ be a finite collection of subgroups generating a family $\mc F$. 
\begin{enumerate}
\item Suppose that $\mc P$ is an almost malnormal collection. 
If $G$ is $FP_{2,\mc F}$ and $\FV_{G, \mc F}(n)$ has linear growth, then $G$ is hyperbolic relative to $\mc P$.
\item Suppose that $\mc F$ contains all finite subgroups of $G$. If $G$ is hyperbolic relative to $\mc P$ then $G$ is $F_{2,\mc F}$ and $\FV_{G, \mc F}(n)$ has linear growth.
\end{enumerate}
\end{theorem}

This result  could be interpreted as an algebraic version of~\cite[Prop. 2.50]{GrMa09} or~\cite[Theorem 1.8]{MP15}. The proof of Theorem~\ref{thmi:rhypchar0}(1) is given in Section~\ref{sec:relhyp}. Theorem~\ref{thmi:rhypchar0}(2) is a result of Przytycki and the author~\cite[Corollary 1.5]{MaPr16}.  

\begin{remark}
The assumption that $\mc P$ is an almost malnormal collection in Theorem~\ref{thmi:rhypchar0}(1)  can not be removed from the statement. Let $G=\Z\oplus \Z$, let $P=\Z\oplus 0$, and let $\mc F$ be the family of subgroups of $P$.  The real line is a cocompact model for $E_{\mc F}G$ when considering the action on which $\Z\oplus 0$ acts trivially and $0\oplus \Z$ acts by translations. In this case, $G$ is $FP_{2, \mc F}$, the function $\FV_{G,\mc F}$ is trivial, and $G$ is not hyperbolic relative to $\mc P$.
\end{remark}

\subsection{The proof of Theorem~\ref{thmi:main}. }

The main step in the proof of Theorem~\ref{thmi:main} is the following subgroup theorem for homological Dehn functions. The proof of Theorem~\ref{thmi:subFV} will be given in Section~\ref{subsec:mainstep}.

\begin{theorem}\label{thm:subFV}\label{thmi:subFV} Let $G$ be a group and $\mc F$ a family of subgroups. Suppose that $G$ admits a cocompact $E_{\mc F}G$  and  $\cdF(G) \leq 2$. If $H \leq G$ is of type $FP_{2, \mc F \cap H}$, then $\FV_{H, \mc F \cap H} \preceq \FV_{G, \mc F}.$
\end{theorem}

\begin{remark}
In the case that $\mc F$ is the trivial family, a version of this result is implicit in the work of Gersten~\cite{Ge96}. Another version in the case that $\mc F$ is the trivial family  is~\cite[Theorem 1.1]{HaMa15} which assumes that there is a $2$-dimensional compact $K(G,1)$ and $H$ is finitely presented. Recall that finite presentability is not equivalent to $FP_2$, see~\cite[Example 6(3)]{BeBr97}.  
\end{remark}

\begin{theorem}\cite[Theorem 1.1]{MaPr16}\label{thm:dismantlable}
Let $G$ be a hyperbolic group relative to a family $\mc F$ containing all finite subgroups. Then there is a cocompact model for $E_{\mc F}G$.
\end{theorem}

\begin{proof}[Proof of  Theorem~\ref{thmi:main}]
By Theorems~\ref{thm:rhypchar0}(2) and~\ref{thm:dismantlable}, the function   $\FV_{G, \mc F}(n)$ has linear growth, and   that $G$ admits a cocompact $E_{\mc F}G$, respectively.  Theorem~\ref{thmi:subFV} implies that $\FV_{H, \mc F\cap H}  \preceq \FV_{G, \mc F}$, and hence   $\FV_{H, \mc F\cap H}$  has linear growth. Since $H$ is $FP_{0, \mc F\cap H}$, the family $\mc F\cap H$ is generated by a finite collection $\mc Q$ of subgroups~\cite[Prop. 3.6.1]{Jg11}.  Assume that
the collection $\mc Q$ has minimal size; observe that no pair of subgroups of $\mc Q$ are conjugate in $H$. Since the collection of maximal parabolic subgroups of $G$ is almost malnormal~\cite[Lemma 2.2]{MaWi10}, then $\mc Q$ is an almost malnormal collection of subgroups in $H$. By Theorem~\ref{thmi:rhypchar0}(1),    $H$ is hyperbolic relative to $\mc Q$.
\end{proof}

\subsection{Outline.}  The rest of the article is organized as follows.  Section~\ref{sec:preliminaries} contains preliminary material including a brief review of Bredon modules. Section~\ref{sec:finiteness} contains a combinatorial proof that $FP_{1, \mc  F}$ is equivalent to $F_{1, \mc F}$, an expected result for which we could not find a proof in the literature.  Section~\ref{sec:InvarianceFVG} introduces the notion of homological Dehn function $\FV_{G, \mc F}$ for a pair $(G, \mc F)$ and discuss its  algebraic and topological approaches.  Section~\ref{sec:subgroupthm} contains the proof of Theorem~\ref{thmi:subFV}. Section~\ref{sec:relhyp} discuss relative hyperbolicity and contains the proofs of Theorem~\ref{propi:smallcancelationgd} and  Theorem~\ref{thmi:rhypchar0}.

\subsection*{Acknowledgments.} The author thanks Gaelan Hanlon, Tomaz Prytu{\l}a and Mario Velasquez for comments.  
We would also  like to thank Dieter Degrijse for pointing out  an error in an earlier version of the article, Ian Leary for consultation at the \emph{Conference on Finiteness Conditions in Topology and Algebra 2015} and Denis Osin for a conversation suggesting the example in Remark~\ref{ex:Osin} at the 
 \emph{Geometric and Asymptotic Group Theory with Applications Conference 2016}. We specially thank the referees for suggestions, feedback, and pointing out necessary corrections. The author acknowledges funding by the Natural Sciences and Engineering Research Council of Canada, NSERC.
 

\section{Preliminaries }\label{sec:preliminaries}

\subsection{$G$-spaces and $G$-maps}

In this article, all group actions are from the left, all spaces are combinatorial complexes and all maps between complexes are combinatorial.  All group actions on complexes are by combinatorial maps. Moreover, all group actions are assumed to have no inversions;  this means that if a cell is fixed setwise, then it is fixed pointwise.  A $G$-complex $X$ is cocompact  if there are finitely many orbits of cells. For a $G$-complex $X$, the $G$-stabilizer of a cell $\sigma$ is denoted by $G_\sigma$, and the fixed point set of  $H\leq G$ is denoted by $X^H$. Since there are no inversions, the fixed point set $X^H$ is a subcomplex of $X$.  A path $I\to X$ in a complex $X$ is a combinatorial map from an oriented subdivided interval; in particular a path has an initial vertex and a terminal vertex.

\begin{definition}[Combinatorial Complexes and Maps]\cite[Chapter I Appendix]{BrHa99}
 A  map $X\to Y$ between CW-complexes is combinatorial if its restriction to each open cell of $X$ is a homeomorphism onto an open cell of $Y$.   A CW-complex $X$  is combinatorial provided  that  the  attaching  map  of  each  open  cell  is  combinatorial  for  a  suitable subdivision.
\end{definition}

\subsection{Modules over the Orbit Category}

Let $\mc F$ be a family of subgroups of $G$. The \emph{orbit category $\mc{O_F}(G)$} is the small category whose objects are the $G$-spaces of left cosets $G/H$ with $H\in \mc F$  and morphisms are the $G$-maps among them.  The set of morphisms from $G/K$ to $G/L$ is denoted by $[G/K, G/L]$. 
 
The \emph{category of right $\Or$-modules} is the functor category of contravariant functors from $\Or$ into the category of abelian groups $\text{AB}$. More concretely, a \emph{right module over $\Or$} is a contravariant functor $\Or \to \text{AB}$, and a morphism $M\to N$ between right  modules over $\Or$ is a natural transformation from $M$ to  $N$.  For a $\Or$-module $M$ and $H\in \mc F$, we denote by $M[G/H]$ the image of the object $G/H$ by the functor $M$.  The abelian group $M[G/H]$ is called the \emph{component of $M$ at $G/H$}. 
 For the rest of this article by an $\Or$-module we shall mean a right $\Or$-module.  
  
Since the category of $\Or$-modules is a functor category, several constructions of the category of abelian groups  $\text{AB}$ can be carried out in $\Or$--modules by performing them componentwise. As a consequence it is a complete, cocomplete, abelian category with enough projectives.    In particular, computing kernels, images, and  intersections  reduces to componentwise computations in $\text{AB}$. For instance, a sequence $K\to L \to M$ of $\Or$--modules is exact if and only if for every $H\in \mc F$ the sequence $K[G/H]\to L[G/H] \to M[G/H]$ of abelian groups is exact.  

\subsection{Free $\Or$-modules}
 
Let $S$ be a  (left) $G$-set. Consider the  contravariant functor  $\mc{O_F}(G) \overset{S}\to \text{SET}$ that maps the object $G/H$ to the fixedpoint set $S^H$ and maps a  morphism $G/H \overset\varphi\to G/K$ given by $\varphi (H)=gK$ to the function $X^K\to X^H$ given by $x \mapsto g.x$. The \emph{$\Or$-module induced by $S$},  denoted by $\Z[\cdot, S]$, is the composed functor   $\mc{O_F}(G) \overset{s}\to \text{SET} \to \text{AB}$ where the second functor is the free abelian group construction $X\mapsto \Z[X]$. For a subgroup $K\leq G$, the module induced by the $G$-set $G/K$ is denoted by $\Z[\cdot , G/K]$. 

The module $\Z[\cdot, G/G]$ is denoted by $\Z_{G, \mc F}$. Note that each component $\Z[G/K, G/G]$ is an infinite cyclic group with  canonical generator the only element of $G/G$. 

An  $\Or$-module $F$ is \emph{free} if $F$ is isomorphic to $\Z [\cdot ,S]$ for $S$ a $G$-set with isotropy groups in $\mc F$. If, in addition, the $G$-set $S$ has finitely many orbits then $\Z [\cdot ,S]$ is a \emph{finitely generated free $\Or$-module}.

\subsection{Augmentation maps}

Let $S$ be a $G$-set and consider the $\Or$-module $\Z[\cdot,S]$. For $K\in \mc F$,  the component $\Z[G/K, S]$ is by definition the free abelian group $\Z[S^K]$ with $S^K$ as a free basis. 
 The \emph{augmentation map of $\Z[\cdot, S]$} is the morphism $\epsilon\colon \Z[\cdot, S] \to \ZGF$ whose components $\epsilon_{G/K}\colon \Z[S^K] \to \Z$ are defined by mapping each element of $S^K$ to the canonical generator of $\Z[G/K, G/G]$.  The main observation is that $\epsilon$ is surjective if and only if the fixed point set $S^K$ is not empty for every $K\in \mc F$.

\subsection{Schanuel's Lemma}

The following proposition is a well known result of homological algebra that is used several times through the article.

\begin{proposition}[Schanuel's Lemma]\cite[Ch.VIII, 4.2]{Brown}\label{Schanuel} Let 
\[0 \rightarrow P_n \rightarrow P_{n-1} \rightarrow \dots \rightarrow P_0 \rightarrow M \rightarrow 0\] and 
\[0 \rightarrow P_n' \rightarrow P_{n-1}' \rightarrow \dots \rightarrow P_0' \rightarrow M \rightarrow 0\] be exact sequences of $\mc{O_F}(G)$--modules such that $P_i$ and $P_i'$ are projective for $i \leq n-1$. 
\begin{itemize}
\item  If $n$ is even then 
\[P_0 \oplus P_1' \oplus P_2 \oplus P_3' \oplus \dots \oplus P_n \ \large \simeq \ P_0' \oplus P_1 \oplus P_2' \oplus P_3 \oplus \dots \oplus P_n' \]
\item If $n$ is odd then 
\[P_0 \oplus P_1' \oplus P_2 \oplus P_3' \oplus \dots \oplus P_n' \ \large \simeq \ P_0' \oplus P_1 \oplus P_2' \oplus P_3 \oplus \dots \oplus P_n. \]
\end{itemize}
\end{proposition}

\subsection{The cellular $\mc{O_F}(G)$-chain complex of a $G$--CW-complex}

Let $G$ be a group, let $\mc F$ be a family of subgroups, and let $X$ be a $G$-complex. These data induce a canonical contravariant functor \begin{equation}\nonumber \mc{O_F}(G)  \longrightarrow \text{TOP} \end{equation} that maps the object $G/H$ to the fixed point set $X^H=\{x\in X \colon\forall h\in H\ h.x=x\}$, and maps a $G$-map $G/H \overset\varphi\to G/K$ given by $\varphi (H)=gK$ to the cellular map $X^K\to X^H$ given by $x \mapsto g.x$. Observe that $X^H$ is a subcomplex of $X$ since we are assuming that our actions are cellular and without inversions.  The composition of the contravariant functor $\mc{O_F}(G)  \longrightarrow \text{TOP}$ with any covariant functor $\text{TOP} \to \text{AB}$ (as a homological functor) is a right $\mc{O_F}(G)$-module.  

 The \emph{augmented cellular $\mc{O_{F}}(G)$--chain complex of $X$}, denoted by $C_*\left(X, \mc{O_{F}}(G)\right)$, is the contravariant functor
\[\begin{array}{ cccccc}
 \mc{O_{F}}(G) & \longrightarrow & \text{TOP} & \longrightarrow  &\text{Chain--Complexes}  \\    \\
 G/H & \mapsto & X^H & \mapsto & C_*(X^H, \Z),
\end{array}
\]
where the second functor is the standard cellular augmented  $\Z$--chain complex functor. 

\begin{remark}\label{rem:FreeChainCpx}
If $X$ is a cocompact $G$--cell-complex such that $X^K$ is $n$-acyclic for every $K\in \mc F$  and all the isotropy groups of $X$ belong to $\mc F$, then  the induced chain complex of $\Or$-modules
\begin{equation}\label{eq:orcc}  C_n(X, \Or) \overset{\partial_n} \longrightarrow \dots \overset{\partial_{2}} \longrightarrow C_1(X, \Or) \overset{\partial_{1}} \longrightarrow C_0(X, \Or) \overset{\epsilon} \longrightarrow \ZGF  \rightarrow 0,\end{equation}
is an exact sequence  of finitely generated free $\Or$-modules.
\end{remark}


\section{Finiteness properties } \label{sec:finiteness}

 Let $G$ be a group and let $\mc F$ be a family of subgroups. 

\begin{definition}  \label{def:fn} 
The group $G$ is of \emph{type $F_{n, \mc F}$} if there is a cocompact $G$--complex $X$ such that all the isotropy groups of $X$ belong to $\mc F$, and for every $K\in \mc F$ the fixed point set $X^K$ is a non-empty and $(n-1)$-connected subcomplex of $X$. A complex $X$ satisfying these conditions is called an \emph{an $F_{n, \mc F}$-complex for $G$}. The group $G$ is of type $F_{\mc F}$ if there is a cocompact model for $E_{\mc F}(G)$.
\end{definition}

\begin{definition}
 \label{def:fpn}
The group $G$ is of \emph{type $FP_{n, \mc F}$} if there is a resolution of $\mc{O_F}(G)$--modules 
\[ \cdots \to P_n \overset{\partial_n} \longrightarrow \dots \overset{\partial_{2}} \longrightarrow P_1 \overset{\partial_{1}} \longrightarrow P_0 \longrightarrow \mathbb Z \rightarrow 0,\] 
such that for $0\leq i\leq n$ the module $P_i$ is  finitely generated and  projective. Such a resolution is called an \emph{$FP_{n, \mc F}$--resolution}.   It is equivalent to require that for $0\leq i\leq n$, the module $P_i$ is finitely generated and free, see for example~\cite[Ch.VIII 4.3]{Brown}. 
\end{definition}

 The main result of the section is the following proposition whose proof, as well as the proof of its corollary, are postponed to  Subsection~\ref{subsec:proofFP1F1}.
 \begin{proposition}\label{prop:FP1Graph} 
If $G$ is $FP_{1, \mc F}$  then $G$ is $F_{1, \mc F}$.
\end{proposition}

\begin{definition}  
 The \emph{family generated by a collection of subgroups $\mc P$ of $G$} consists of all subgroups $K\leq G$ such that that $g^{-1}Kg\leq P$ for some  $g\in G$ and $P\in \mc P$. 
\end{definition}

\begin{corollary}\label{cor:FinEdSt}
If $G$ is $FP_{1, \mc F}$ and  $\mc F$ is a family generated by a finite and almost malnormal collection $\mc P$, then there is an $F_{1, \mc F}$-graph with finite edge   $G$-stabilizers.
\end{corollary}

The second part of the section will recall the definition of the coned-off Cayley complex of a relative presentation, and observe that these complexes provide $F_{2, \mc F}$-complexes in some cases.

 \subsection{Proof of Proposition~\ref{prop:FP1Graph} } \label{subsec:proofFP1F1}

Take an $FP_{1, \mc F}$--resolution consisting of only free $\Or$-modules, see for example~\cite[Ch.VIII 4.3]{Brown},  
\[\label{eq:freeresolution}  
\Z[\cdot, T]   \overset p \longrightarrow \Z[\cdot, V] \overset\epsilon\longrightarrow \mathbb \ZGF \rightarrow 0\]
where  $\Z[\cdot, V]\overset\epsilon\to \ZGF$ is the augmentation map.  Then $T$ and $V$ are $G$-sets with isotropy groups in $\mc F$ and finitely many $G$-orbits.  For each $K\in \mc F$ the sequence of abelian groups arising by taking the $G/K$-component 
\begin{equation}\nonumber  \Z[T^K]  \overset p \longrightarrow \Z[V^K] \overset\epsilon\longrightarrow \mathbb \Z \rightarrow 0\end{equation}
is exact.  
 
\begin{step}
\label{part1} For each $t\in T$ there is a finite directed $G_t$-graph $\Gamma_t=(V_t, E_t)$ such that $V_t$ is a subset of $V$ and
\begin{enumerate}
\item $\Gamma_t$  is connected when considered as an undirected graph,
\item the $G_t$-action on $V_t$ is the restriction of the $G$-action on $V$,
\item the $G_t$-stabilizer of an edge $e\in E_t$ is the intersection of the $G_t$-stabilizers of its endpoints, and
\item  $\sum_{e\in E_t} (e_+ - e_-)  = p(t) .$
\end{enumerate}
\end{step}
\begin{proof} 
Since $T$ has isotropy groups in $\mc F$, we have that $G_t\in \mc F$ and hence    the sequence 
\begin{equation}\nonumber  \Z[T^{G_t}]  \overset  p \longrightarrow \Z[V^{G_t}] \overset\epsilon \longrightarrow \mathbb \Z \rightarrow 0\end{equation}
is exact. Therefore the fixed point set $V^{G_t}$ is non-empty.  
Choose $\star_t\in V^{G_t}$.  Consider the expression   
 \[ p(t) =\sum_{v\in V} n_{t, v}v,\]
where each $n_{t,v}$ is an integer. Define 
$W_t = \{v \in V \colon n_{t, v}\neq 0\}.$ 
Define 
\[V_t = W_t\cup \{\star_t\},\]
and observe that  $V_t$ is a finite $G_t$-invariant subset of $V$.  The edge set $E_t$ is defined as the following subset of $V_t\times \N$,
\[  E_t = \left\{ v\times j \colon  v\in V_t \text{ and } 1\leq j \leq |n_{t, v}| \ \right\}.\]
For $e=v\times j$ define $e_+=v$ and $e_-=\star_t$ if $n_{t,v}>0$, and 
$e_+=\star_t$ and $e_-=v$ if $n_{t,v}<0$.   Since each edge of $E_t$ has $\star_t$ as one of its endpoints, $\Gamma_t = (V_t, E_t)$ is connected (when considered as a $1$-complex). 
The $G_t$-action on $V_t$ induces an action on $E_t$ by 
\[ g.(v\times j) = g.v \times j\]
respecting the adjacency relation. Since $\epsilon \circ p (t) =0$ and $\epsilon$ is the augmentation map, we have that $\sum_{v\in V_t} n_{t,v} = 0$. Therefore, 
\[ \sum_{e\in E_t} (e_+ - e_-) =
 \sum_{v\in V_t}  n_{t,v} \left(v-\star_t\right) =  p(t). 
\qedhere\]
 \end{proof}

\begin{step}\label{lem:fromHsetstoGset}
Let $K$ be a subgroup of $G$, and let $E$ be a $K$-set. 
Then there is $G$-set $\hat E$ and an injective and $K$-equivariant map $\imath\colon E \hookrightarrow \hat E$ with the following properties: 
\begin{enumerate}
\item The map $\imath$ induces a bijection of orbit spaces $E/K \to \hat E/G$.
\item For each $e\in E$, the $K$-stabilizer $K_e$ equals the $G$-stabilizer $G_{\imath(e)}$.
\item If $V$ is a $G$-set and $f\colon E \to V$ is  $K$-equivariant, then there is a unique $G$-map $\hat f\colon \hat E \to V$ such that $\hat f \circ \imath = f$.
\end{enumerate}
\end{step}
\begin{proof}
 Observe that $E$ is isomorphic to the $K$-space $\coprod_{i\in I} K/K_i$ where $I$ is a complete set of representaives of $K$-orbits of $E$ and $K/K_i$ denotes the $K$-space  of left cosets of $K_i$ in $K$. Define $\hat E$ as $\coprod_{i\in I} G/K_i$ and observe that the assignment $K_i \mapsto K_i$ induces an injective $K$-map $E \hookrightarrow \hat E$. The three statements of the lemma are observations.  
\end{proof}

\begin{step}\label{part2} There is a directed $G$-graph $\Gamma$ with vertex set $V$ with the following properties.
\begin{enumerate}
\item For each $K\in \mc F$, the fixed point set $\Gamma^K$ is a non-empty connected subgraph of $\Gamma$. 
\item $\Gamma$ has finitely many orbits of vertices and edges. 
\item All isotropy groups of $\Gamma$ belong to $\mc F$.
\end{enumerate} 
In particular, as a complex, $\Gamma$ is an $F_{1,\mc F}$-space for $G$.
\end{step}
\begin{proof}
Let $\Gamma$ be the directed $G$-graph whose vertex set is the $G$-set $V$, and edge $G$-set $E$ is defined as follows.

Let $R_T$ be a complete set of representatives of $G$-orbits of $T$. 
For $t\in R_T$, let $\Gamma_t = (V_t, E_t)$ be the directed $G_t$-graph of Step~\ref{part1}. 
Let $\hat E_t$ be the $G$-set provided by Step~\ref{lem:fromHsetstoGset} for the $G_t$-set $E_t$. 
The $G_t$-maps $(-)_t\colon E_t \to V$ and $(+)_t\colon E_t \to V$ defining initial and terminal vertices extend to $G$-maps  $(-)_t\colon \hat E_t \to V$ and $(+)_t\colon \hat E_t \to V$.
Let $E$ be the disjoint union of $G$-sets $E=\coprod_{t\in R_T} \hat E_t$ and let $(+)\colon E \to V$ and $(-)\colon E\to V$ be the induced maps.  Note that $E$ has finitely many $G$-orbits.

By construction, the group $G$ acts on the directed graph $\Gamma=(V,E)$ with finitely many orbits of vertices and edges. Moreover, the $G$-action on $\Gamma$ as a complex has no inversions (if an edge is fixed setwise, the it is fixed pointwise).   
The stabilizer $G_e$ of an edge $e\in E$ is the intersection of two subgroups in $\mc F$, hence $G_e \in \mc F$. It remains to show that the fixed point sets $\Gamma^K$ are non-empty and connected for $K\in\mc F$.

Let $h\colon \Z[T] \to \Z[E]$ be the morphism of $\ZG$-modules given  by $t\mapsto \sum_{e\in E_t} e$ for $t\in R_T$. Note that $h$ is well defined since   the subgroup $G_t$ preserves $E_t$ setwise. Since 
\[ \sum_{e\in E_t} (e_+ - e_-)  = p(t) \]
by definition of $E_t$, it follows that for each $K\in \mc F$ the following diagram commutes
\begin{equation}\label{ea:diagWE} \xymatrix{
\Z[T^K] \ar@{-->}[d]^h  \ar[r]^{p} & \Z[V^K] \ar[r]^{\epsilon} \ar[d]^{Id} & \Z  \ar[r] & 0 \\
 \Z[E^K] \ar[r]^{\partial} & \Z[V^K]   &&}
\end{equation}
where  $\partial$ is the standard boundary map from $1$-chains to $0$-chains of $\Gamma$. Since the top row is an exact sequence, it follows that $H_0(\Gamma^K, \Z) \cong \Z$  and hence $\Gamma^K$ is non-empty and connected.
\end{proof}
 
\begin{proof}[Proof of Corollary~\ref{cor:FinEdSt}]
Consider the $G$-set  $V=\coprod_{P\in \mc P} G/P$.  Observe that  $V/G$ is finite,  every isotropy group of $V$ is in $\mc F$, and for every $K\in \mc F$ the fixed point set $V^K$ is non-empty. The first two properties imply that $\Z[\cdot, V]$ is a finitely generated free $\Or$-module, and the third property implies that the augmentation map $\Z[\cdot , V] \overset \epsilon \to \Z_{G,\mc F}$ is surjective. Since $G$ is $FP_{1, \mc F}$, there is a partial $FP_{1,\mc F}$-resolution $\Z[\cdot, T] \overset p \to \Z[\cdot, V] \overset \epsilon \to \Z_{G,\mc F} \to 0$. Then the proof of Proposition~\ref{prop:FP1Graph} provides an $F_{1,\mc F}$-graph $\Gamma$ with vertex set $V$. Since $\mc P$ is an almost malnormal collection, the intersection of the $G$-stabilizers of two distinct vertices is finite, and hence stabilizers of edges of $\Gamma$ are finite. 
\end{proof}

\subsection{Coned-off Cayley complexes and $F_{2,\mc F}$}

For this subsection, let $G$ be a group, let $\mc P$ be a collection of subgroups of $G$, and let $\mc F$  be the family generated by $\mc P$.

\begin{definition}[Finite relative presentations]\cite{Os06} The group $G$ is \emph{finitely presented relative to $\mc P$ if there is finite subset $S$ of $G$,  and a finite subset $R$ of words over the alphabet $S\sqcup \bigsqcup_{P\in \mc P} P$  such that the natural homomorphism from the  the free product $F=(\ast_{P\in \mc P} P)\ast F(S)$}  into $G$  is surjective and its kernel is normally generated by the elements of $R$ (considered as elements of $F$); here $F(S)$ denotes the free group on the set of letters $S$.  In this case, the data $\langle S, \mc P | \mc R \rangle$ is called a \emph{finite relative presentation of $G$}. 
\end{definition}

\begin{definition}[Cayley graph]
Let $S$ be a subset of $G$ closed under inverses. The \emph{directed Cayley graph} $\Gamma=\Gamma(G, S)$ is  the directed graph with vertex set $G$ and edge set define as follows: for each $g\in G$ and $s\in S$, there is an edge from $g$ to $gs$ labelled $s$.   It is an observation that if $S$ generates $G$ then $\Gamma$ is connected, and that if $S$ is finite then the quotient $\Gamma/G$ is a finite graph. 

The \emph{undirected Cayley graph} is obtained from $\Gamma(G, S)$ by first ignoring the orientation of the edges, and then identifying edges between the same pair vertices. Observe that any combinatorial path in the undirected Cayley graph is completely determined by its starting vertex and a word in the alphabet $S$. If the path is locally embedded, the corresponding word is reduced.  From here on, we will only consider the undirected Cayley graph which we will denote by $\Gamma(G, S)$, and we will call it \emph{the Cayley graph of $G$ with respect to $S$}.
\end{definition}

The definition of coned-off Cayley complex as well as Lemma~\ref{lem:248} below are based on~\cite[Definition 2.47, Lemma 2.48]{GrMa09} respectively. In the cited article, the definition of coned-off Cayley complex has some minor differences, and the constructions assume that each $P\in \mc P$ is finitely generated.  We do not require this last assumption for our arguments

\begin{definition}[Coned-off Cayley Graph and Coned-off Cayley Complex] \label{def:conedoffCC}
Let $\langle S, \mc P |  \mc R \rangle$ be a finite relative presentation of $G$.

The \emph{coned-off Cayley graph $\hat \Gamma$  of $G$ relative to $\mc P$ and $S$} is the $G$-graph obtained from the Cayley graph of $G$ with respect to $S\cup S^{-1}$,  by adding  a new vertex $v(gP)$ for each left coset $gP$ with $g\in G$ and $P\in \mc P$, and edges from $v(gP)$ to each element of $gP$.  The vertices of the form $v(gP)$ are called \emph{cone-vertices}, and any other vertex a \emph{non-cone vertex}. The $G$-action on the cone-vertices is defined using the  $G$-action on the corresponding left cosets, and this action extends naturally to the edge set. 

Observe that each  locally embedded path in $\hat \Gamma$ between elements of $G$ has a unique \emph{label} corresponding to a word in the alphabet $S\sqcup \bigsqcup_{\mc P} P$, since edges between vertices in $G$ have a label in $S\cup S^{-1}$, and for each oriented and embedded path of length two with middle vertex a cone point $v(gP)$ corresponds a unique element of $P$. The label of a path that consists of a single non-cone vertex is defined as the empty word.

The \emph{coned-off Cayley complex $X$} induced by $\langle S, \mc P |  \mc R \rangle$  is the $2$-complex obtained by equivariantly attaching $2$-cells to the coned-off Cayley graph $\hat \Gamma$ as follows.   
Attach a $2$-cell with trivial stabilizer to each loop of $\hat \Gamma$ corresponding to a relator $r \in \mc R$ in a manner equivariant under the $G$-action on $\hat \Gamma$. 
\end{definition}

\begin{lemma} \label{lem:248}
Let $\langle S, \mc P | \mc R \rangle$ be a finite relative presentation of $G$. The   induced coned-off Cayley complex $X$ is simply-connected.
\end{lemma}
\begin{proof}
Let $T$ be the coned-off Cayley graph of $F=(\ast_{P\in \mc P} P)\ast F(S)$ with respect to $\mc P$ and $S$. Observe that $T$ is simply-connected since it is a tree. 
Let $\hat \Gamma$ be the coned-off Cayley graph of $G$ with respect to $\mc P$ and $S$.
Observe that  $\hat \Gamma$ is recovered as the quotient $T/N$.  The main observation is that $T \to \hat \Gamma$ is a normal covering space with $N$ as a group of deck transformations. Indeed, the edge stabilizers of $T$ with respect to $N$ are trivial by definition; the vertex stabilizers of $T$ with respect to $N$ are trivial since each factor $P$ of $F$ embeds in $G$ an therefore $N\cap P$ is trivial for every $P\in \mc P$.  It follows that the $1$-skeleton of $X$ has fundamental group isomorphic to $N$, and since $N$ is the normal closure of $\{r_1, \ldots, r_m\}$ in $F$ the definition of $X$ implies that $\pi_1(X)$ is trivial. 
\end{proof}

\begin{proposition}\label{prop:F2Fchar}
Suppose $\mc P$ is a finite and  malnormal collection of subgroups of $G$, and  $\langle S, \mc P |  \mc R \rangle$ is a finite relative presentation of $G$. Then the induced coned-off Cayley complex $X$ is an $F_{2, \mc F}$-space with trivial $1$-cell $G$-stabilizers. In particular, for each $K\in \mc F$ the fixed point set $X^K$ is either a point or $X$. 
\end{proposition}
\begin{proof}
The space $X$ is simply-connected by~\ref{lem:248}, and since the relative presentation is finite, the $G$-action on $X$ is cocompact. Let $K\in \mc F$ be a non-trivial subgroup. Since the $G$-stabilizers of $1$-cells and $2$-cells of $X$ are trivial,  the fixed point set $X^K$ consists only of $0$-cells. Since $\mc P$ is a malnormal collection and  $K$ fixes at least one cone-vertex, it follows that $X^K$ is a single cone-vertex of $X$. Therefore $X$ is  an $F_{2, \mc F}$-space with trivial $1$-cell stabilizers.
\end{proof}


\section{The relative homological Dehn function}\label{sec:InvarianceFVG}

Thoughtout this section, $G$ is a group,  $\mc F$ is a family of subgroups, and $G$ is $FP_{2, \mc F}$.

\begin{definition} \label{def:based}
If $F$ is a  free $\ZG$--module with $\ZG-$basis $B$, then $\left \{ gb \colon g \in G , b\in B \right \}$ is a basis for $F$ as a free $\mathbb Z$-module inducing a $G$-equivariant $\ell_1$-norm $\| \cdot \|_1$. 
We call a free $\ZG-$module \emph{based} if it is understood to have a fixed $\ZG$-basis and we use this basis for the induced $\ell_1$--norm $\| \cdot \|_1$. 
\end{definition}

\begin{definition}  \label{def:filling-norm}
Let $\eta\colon F\to M$ be a surjective homomorphism of $\ZG$-modules and suppose that $F$ is free, finitely generated, and based. The \emph{filling  norm on $M$ induced by $\eta$ and the based free module $F$} is the $G$-equivariant norm defined by
\[\| m \|_{\eta} =  \min \left \{ \  \| x \|_1 \colon x \in F , \ \eta(x) = m \right \}.\] 
\end{definition}

\begin{definition}[Algebraic Definition of Relative Homological Filling Function]  \label{def:algFVG}  The \emph{homological filling function} of $G$ relative to $\mc F$, is the  function \[\FV_{G, \mc F}\colon \N \to \N\] defined as follows. 
Let $P_\ast$  be a resolution of $\mc{O_F}(G)$--modules for $\ZGF$ of type $FP_{2, \mc F}$ and let 
\[ P_{2} \overset{\partial_{2}} \longrightarrow P_1 \overset{\partial_{1}} \longrightarrow P_0 \longrightarrow \mathbb Z \rightarrow 0\] 
be the resolution of finitely generated $\ZG$--modules obtained by evaluating $P_\ast$ at $G/1$. Choose filling norms for the $\ZG$--modules  $P_1$ and $P_{2}$,  denoted by $\|\cdot \|_{P_1}$ and $\|\cdot \|_{P_{2}}$ respectively. 
 Then
\[\FV_{G,   \mc F}(k) = \max \left \{ \ \| \gamma \|_{\partial_{2}} \colon \gamma \in \ker(\partial_1), \ \| \gamma \|_{P_1} \leq k \ \right \}\]
where $ \| \cdot \|_{\partial_{2}}$ is the filling norm on $\ker(\partial_1)$ given by 
\[ \| \gamma \|_{\partial_{2}} = \min \left\{  \| \mu\|_{P_{2}} \colon \mu \in P_{2},\  \partial_{2} (\mu) =\gamma \right\} .\]
\end{definition}

\begin{theorem} \label{AlgebraicDef} The growth rate of  $\FV_{G, \mc F}$ is an invariant of the pair $(G, \mc F)$.
\end{theorem}

The rest of the section is divided into two parts, first the proof of Theorem~\ref{AlgebraicDef}, and then a discussion on a topological approach to $\FV_{G, \mc F}$.

\subsection{Proof of Theorem~\ref{AlgebraicDef}}

The proof in the case that $\mc F$  consists only of the trivial subgroup was proved in~\cite[Theorem  3.5]{HaMa15}. The argument below  is   essentially a  translation of the argument in~\cite[Proof of Theorem  3.5]{HaMa15}  from the category of $\ZG$-modules  to the category of $\Or$-modules up to using Lemma~\ref{lem:ProjBounded} instead of~\cite[Lemma 2.8]{HaMa15}. 

\begin{lemma} \label{EquivStd}\cite[Lemma 4.1]{Ge96}\cite[Lemma 2.9]{HaMa15} Any two filling norms $\| \cdot \|_\eta$ and $\| \cdot \|_\theta$ on a finitely generated $\ZG$--module $M$ are equivalent in the sense that there exists a constant $C\geq 0$ such that $C^{-1} \| m \|_\eta \leq \| m \|_\theta \leq C \| m \|_\eta$ for all $m \in M$.
\end{lemma}

\begin{lemma} \label{lem:ProjBounded}
Let $\varphi : P \rightarrow Q$ be a homomorphism between finitely generated $\ZG-$modules. Let $\| \cdot \|_P$ and $\| \cdot \|_Q$ denote filling norms on $P$ and $Q$ respectively. There exists a constant $C>0$ such that $\| \varphi(p) \|_Q \leq C  \| p \|_P$ for all $p \in P$.
\end{lemma}
\begin{proof}
Consider the commutative diagram 
\[\xymatrix{ 
A \ar[r]^{\tilde \varphi} \ar[rd]^\psi \ar[d]_\rho & B \ar[d] \\
P \ar[r]_\varphi    & Q
}\]
constructed as follows.  Let $A$ and $B$ be finitely generated and based free $\ZG$-modules,  and let  $A\to P$ and $B\to Q$ be surjective morphisms inducing the filling norms $\| \cdot \|_P$ and $\| \cdot \|_Q$.   Since $A$ is free and $B\to Q$ is surjective, there is a lifting $\tilde \varphi  \colon A \to B$ of $\varphi$.   Since $\tilde \varphi$ is morphism between finitely generated, free, based $\ZG-$modules, by~\cite[Lemma 2.7]{HaMa15}, there is a constant $C$ such that $\|\tilde \varphi (a)\|_1 \leq C \|a\|_1$ for every $a\in A$. Let $p\in P$ and let $a \in A$ such that $\rho (a) =p$. It follows that 
\begin{equation}\nonumber
 \|\varphi (p)\|_Q  \leq \| \tilde \varphi (a) \|_1  \leq C  \| a \|_1.
\end{equation}
Since the above inequality holds for any $a \in A$ with $\rho (a)=p$, it follows that 
\begin{equation}\nonumber
\begin{split}
  \|\varphi (p)\|_Q & \leq C \cdot \min_{\rho (a)= p} \left\{\|a\|_1  \right \}= C \|p\|_P.  \qedhere  
\end{split}
\end{equation}
\end{proof}

\begin{proof}[Proof of Theorem~\ref{AlgebraicDef}] Let $P_\ast$ and $Q_\ast$  be two resolutions of $\mc{O_F}(G)$--modules for $\ZGF$ of type $FP_{2, \mc F}$. For $i=1,2$, let $\|\cdot \|_{P_i}$ and $\|\cdot \|_{Q_i}$ be filling norms of the components $P_i [G/1]$ and $Q_i[G/1]$ respectively.  Denote by $\FV_{G, \mc F}^P$ and $\FV_{G, \mc F}^Q$ the homological filling functions induced by  $P_\ast$ and $Q_\ast$ respectively.  By symmetry, it is enough to prove that there is $C\geq 0$ such that for every $k\in \N$, 
\begin{equation}\label{eq:preceq}  \FV^{Q}_{G, \mc F}(k) \leq C\cdot \FV^{P}_{G, \mc F} ( Ck + C ) + Ck + C. \end{equation}

Define $C$ as follows. 
Since the category of $\mc{O_F}(G)$-modules  is an abelian category with enough projectives, any two projective resolutions of an $\mc{O_F}(G)$--module are chain homotopy equivalent.   Therefore,  there are chain maps $f_i \colon Q_i \rightarrow P_i, \ g_i\colon P_i \rightarrow Q_i$, and a map $h_i : Q_i \rightarrow Q_{i+1}$ such that 
\[ \partial_{i+1} \circ h_i + h_{i-1} \circ \partial_i = g_i \circ f_i - id.\] 
From here on, we only consider the components at $G/1$ of these resolutions and morphisms, so we are in the category of $\ZG$-modules.   We have two exact sequences of finitely generated $\ZG$-modules
\begin{equation}\nonumber
P_2\overset{\delta_2}\to P_1 \overset{\delta_1}\to P_0 \to \Z \to 0,  \quad  \quad Q_2\overset{\partial_2} \to Q_1 \overset{\partial_1}\to Q_0 \to \Z \to 0
\end{equation}
and the morphisms  $f_i\colon Q_i \to P_i$, $g_i\colon P_i \to Q_i$, and $h_i\colon Q_i \to Q_{i+1}$.  
Let $C$ be the maximum of the constants given by Lemma~\ref{lem:ProjBounded} for $g_{2}$,  $h_1$, and $f_1$ and the  filling-norms provided . 

Inequality~\eqref{eq:preceq} is proved exactly as in~\cite[Proof of Theorem 3.5]{HaMa15} by using Lemma~\ref{lem:ProjBounded} instead of~\cite[Lemma 2.8]{HaMa15}. We include the short argument for completeness.

Fix $k$.  Let $\alpha \in \ker( \partial_1)$ be such that $\| \alpha \|_{Q_1} \leq k$. 
Choose $\beta \in P_{2}$ such that $\delta_{2}(\beta) = f_1 (\alpha )$ and $\|f_1(\alpha) \|_{\delta_{2}} = \| \beta \|_{P_{2}}$. Then
\begin{equation}\nonumber
\begin{split}
\alpha    =  g_1\circ f_1 (\alpha) - \partial_{2} \circ h_1(\alpha) + h_{0} \circ \partial_1 (\alpha)  & =  g_1\circ \delta_2 (\beta) - \partial_{2} \circ h_1(\alpha) \\
		& = \partial_2 \circ g_2   (\beta) - \partial_{2} \circ h_1(\alpha) \\
		& = \partial_2 \left( g_2   (\beta) -  h_1(\alpha) \right)
\end{split}
\end{equation}
It follows that
\begin{align} \nonumber
\| \alpha \|_{\partial_{2}}  \leq \left \| g_{2}(\beta) - h_1(\alpha) \right \|_{Q_2} & \leq \left \| g_{2}(\beta)  \right \|_{Q_{2}} + \left \| h_1(\alpha ) \right \|_{Q_{2}}      \\ \nonumber
& \leq C  \cdot \| \beta \|_{P_2} + C \cdot  \| \alpha \|_{Q_1} \\ \nonumber 
&=  C \cdot \left \| f_1( \alpha ) \right \|_{\delta_{2}} + C \cdot \| \alpha \|_{Q_1}   \\  \nonumber
& \leq C \cdot \FV^{P}_{G, \mc F} ( \| f_1(\alpha ) \|_{P_1} ) + C \| \alpha \|_{Q_1}    \\  \nonumber
& \leq C \cdot  \FV^{P}_{G, \mc F} \left ( C \| \alpha \|_{Q_1} \right) + C  \| \alpha \|_{Q_1}    \\  \nonumber
& \leq C \cdot  \FV^{P}_{G, \mc F} \left ( C k \right) + C  k    
\end{align}  
Since $\alpha$ was arbitrary,   inequality~\eqref{eq:preceq} holds for every $k \in \mathbb N$ completing the proof.
\end{proof}

\subsection{Topological approach to $\FV_{G,\mc F}$}

\begin{definition}\label{def:FVX} Let $X$ be a cell complex.  The \emph{homological Dehn function} of $X$ is the function $\FV_{X} \colon\N \to \N\cup \{\infty\}$ defined as 
\[\FV_X (k) = \max \left \{ \ \| \gamma   \|_{\partial}  \colon \gamma \in Z_1(X, \Z), \ \| \gamma \|_1 \leq k \ \right \},\]
where 
\[ \| \gamma  \|_\partial  = \min \left \{ \ \| \mu \|_1 \colon \mu \in C_{2}(X,\Z), \ \partial ( \mu ) = \gamma \ \right \},\]
and the maximum and minimum of the empty set are defined as zero and $\infty$ respectively. 
\end{definition}

Recall that a cell complex $X$ is $n$-acyclic if its reduced homology groups with integer coefficients $\tilde H_k(X, \Z)$ are trivial for $0\leq k\leq n$. 
  
 \begin{definition}
 \label{def:FH}
The group $G$ is $FH_{n, \mc F}$ if there is a cocompact $G$--complex $X$ such that all isotropy groups of $X$ belong to $\mc F$ and for every  $K\in \mc F$ the fixed point set $X^K$ is a non-empty $(n-1)$-acyclic subcomplex of $X$. Such a complex is called an $FH_{n, \mc F}$-complex.
\end{definition}

\begin{definition}[Topological Definition of Relative Homological Dehn Function]  \label{def:algFVG} Let $X$ be an $FH_{2, \mc F}$-complex for $G$.  The \emph{homological filling function $\FV_{G, \mc F}$} of $G$ relative to $\mc F$, is the (equivalence class of the) homological filling function $\FV_{X}\colon \N \to \N\cup\{\infty\}$.   
\end{definition}

\begin{proposition}\label{cor:algvstop}
Suppose that $G$ is of type $FH_{2, \mc F}$. Then the algebraic and topological definitions of relative homological Dehn function are equivalent. 
\end{proposition}

\begin{proof}
Let $X$ be an $FH_{n, \mc F}$-complex for $G$, see Definition~\ref{def:FH}. It follows that induced augmented cellular chain complex of $\Or$-modules $C_\ast (X, \Or)$ is an $FP_{2, \mc F}$ resolution of $\ZGF$.  By Theorem~\ref{AlgebraicDef}, $\FV_X$ and $\FV_{G, \mc F}$ are equivalent.
\end{proof}

\begin{proposition}\label{prop:lastlemma}
Suppose that $G$ is of type $FP_{2, \mc F}$
\begin{enumerate}
\item If  $X$ is an $F_{1, \mc F}$-graph then there is a $1$-acyclic $F_{1, \mc F}$-space $Y$ with $1$-skeleton $X$ and finitely many $G$-orbits of $2$-dimensional cells.
\item If $Y$ is a $1$-acyclic $F_{1, \mc F}$-space then $\FV_{G, \mc F}$ and $\FV_{Y}$ are equivalent.
\end{enumerate}
\end{proposition}

The proof of Proposition~\ref{prop:lastlemma} requires the following lemma. 

\begin{lemma}\label{lem:last}
If $G$ is of type $FP_{2, \mc F}$ and $X$ is an $F_{1, \mc F}$-space then 
the $\Or$-module of cellular $1$-cycles $Z_1(X, \Or)$ is finitely generated. 
\end{lemma}
\begin{proof}
Consider the exact sequence of $\Or$-modules
\begin{equation}\nonumber  0\to Z_1(X,\Or) \to C_1(X,\Or) \to  C_0(X,\Or) \to \ZGF \to 0.\end{equation}
The assumptions on $X$ imply that each $C_i(X, \Or)$ is finitely generated and free.
Since $G$ is $FP_{2, \mc F}$, there is an exact sequence of $\Or$-modules
\begin{equation}\nonumber  0\to K \to P_1\to P_0 \to \ZGF \to 0\end{equation} where each $P_i$ is finitely generated and projective, and $K$ is finitely generated.   Applying Schanuel lemma~\ref{Schanuel} to these sequences shows that $Z_1(X, \Or)$ is finitely generated.
\end{proof}

\begin{proof}[Proof of Proposition~\ref{prop:lastlemma}]
 We claim that there are finitely many circuits $\gamma_1, \ldots , \gamma_n$ of $X$ such that the corresponding $1$-cycles form a generating set of the $\ZG$-module  $Z_1(X,\Z)$. By Lemma~\ref{lem:last} the $O_{\mc F}(G)$-module of cellular $1$-cycles $Z_1(X,\Or)$ is finitely generated, hence $Z_1(X,\Z)$ is finitely generated. The claim follows by observing that each cellular $1$-cycle  of $X$ can be expressed as a finite sum of $1$-cycles $\sum_{i} \alpha_i$ where each $\alpha_i$ is induced by a circuit~\cite[Lemma A2]{Ge98}.   Let $Y$ by the $2$-dimensional $G$-complex with $1$-skeleton $X$ obtained as follows. 
For each circuit $\gamma_i$, attach to $X$ a $2$-dimensional cell with boundary $\gamma_i$ and trivial $G$-stabilizer, and extend equivariantly. By construction, $Y$ is an $F_{1, \mc F}$-space with trivial $H_1(Y,\Z)$, and finitely many $G$-orbits of $2$-cells.

To prove the second statement,  Lemma~\ref{lem:last} implies that $Z_1(Y, \Or)$ is finitely generated.  Hence there is an $FP_{2,\mc F}$ resolution of the form \[\cdots \to Q_2  \to C_1(Y,\Or)\to C_0(Y,\Or) \to \Z_{G,\mc F} \to 0\] where $Q_2$ is finitely generated and free. The expression defining $\FV_{G, \mc F}$ using this resolution and the definition of $\FV_Y$ differ only on the  chosen filling norm on $Z_1(Y, \Z)$.  Since any two filling norms are equivalent, Lemma~\ref{EquivStd}, we have that $\FV_{G, \mc F}$ and $\FV_Y$ are equivalent.
\end{proof}


\section{Proof of Theorem~\ref{thm:subFV}.}\label{sec:subgroupthm}

The section is divided into two parts. In the first part, we introduce the notion of submodule distortion and prove a technical result relating retractions  and module distortion, Proposition~\ref{thm:retraction}. In the second part, this proposition is used to prove Theorem~\ref{thm:subFV}.  Through the rest of the section $G$ is a group and  $\mc F$ is a family of subgroups.

\subsection{Submodule distortion and retractions}

\begin{definition}
 The \emph{distortion function of a finitely generated $\ZG$-submodule $P$ in a finitely generated $\ZG$-module $L$} is defined by
\[\Dist^L_P(k) = \max \left \{ \ \| \gamma \|_{P} \colon  \ \gamma \in P \text{ and } \| \gamma \|_{L} \leq k \ \right \},\]
where $\|\cdot \|_{P}$ and $\|\cdot \|_{L}$ are filling norms for $P$ and $L$, and 
$\Dist^L_P(k)$ is allowed to be $\infty$.
\end{definition}

\begin{remark}
The equivalence of  filling norms on a finitely generated $\ZG$-module, Lemma~\ref{EquivStd}, implies that the growth rate of $\Dist^L_P(k)$  is independent of the choice of filling norms.
\end{remark}

\begin{remark} \label{rem:distFV}
In the framework of Definition~\ref{def:algFVG}, the homological filling function of $G$ relative to $\mc F$ is the distortion function of $\kernel (\partial_1)$ in $P_1$, specifically,  $\FV_{G, \mc F} = \Dist_{\kernel(\partial_1)}^{P_1}$.
\end{remark}

\begin{definition}\label{def:Ffree}
A $\ZG$-module $M$ is \emph{$\mc F$-free} if there is a $G$-set $S$ with all isotropy groups in $\mc F$ such that  the induced $\ZG$-module $\Z[S]$ is isomorphic to $M$.
\end{definition}

\begin{remark}Let $M$ be a free $\Or$-module.
Then the $G/1$-component $M[G/1]$ is an $\mc F$-free $\ZG$-module.   If $M$ is finitely generated as an $\Or$-module, then $M[G/1]$ is finitely generated as a $\ZG$-module.
\end{remark}

\begin{proposition}\label{thm:retraction}
Let $H$ be a subgroup of $G$. 
Consider a commutative diagram 
 \begin{equation}\nonumber \xymatrix{ 
M  \ar[rr] &   & L \\
Q \ar[rr] \ar[u] \ar[rr]_{\imath}&    &P\ar[u] 
}\end{equation}
where 
\begin{itemize}
\item $L$ and $P$ are finitely generated $\ZG$-modules and $P$ is a submodule of $L$,
\item   $M$ and $Q$ are finitely generated $\ZH$-modules and $Q$ is a submodule of $M$,
\item as $\ZH$-modules, $Q$ is a submodule of $P$, and $M$ is a submodule of $L$,  
\item all the arrows in the diagram are inclusions. 
\end{itemize}
Let $\|\cdot\|_L$ and $\|\cdot\|_M$ be filling norms on $L$ and $M$ respectively. Suppose there is an integer $C_0\geq 1$ such that for every $x\in M$
\[ \|x\|_L \leq C_0\cdot \|x\|_M.\] 
Suppose that $P$ is $\mc F$-free and there is a commutative diagram of $\ZH$-modules
 \begin{equation}\nonumber \xymatrix{ 
  & P \ar[rd]^{\rho}  &   \\
  Q \ar[ur]^\imath \ar[rr]_{Id} &  & Q.    			  
}\end{equation}
Then 
\[ \Dist_Q^M \preceq \Dist_P^L .\]
\end{proposition}

\begin{proof}[Proof of Proposition~\ref{thm:retraction}]
  Since $P$ is a finitely generated $\mc F$-free $\ZG$-module, there is a  $G$-set $S$ with isotropy groups in $\mc F$ and finitely many $G$-orbits  such that $P\cong \Z[S]$. 
  Denote by $\|\cdot\|_P$ the $\ell_1$-norm on $P$ induced by the free $\Z$-basis $S$; observe that this norm is a filling norm on $P$ as a $\ZG$-module.

Since $Q$ is a finitely generated $\Z H$-module, there is a partition $S=S_1 \cup S_2$ where each $S_i$ is an $H$-equivariant subset of $S$, the quotient $S_1/H$ is finite, and the inclusion $\imath\colon Q \to P$ factors through the inclusion $P_1 \hookrightarrow P$ as in the commutative diagram below
 \begin{equation}\nonumber \xymatrix{ 
& P_1 \ar[d] \\
Q \ar[ru]  \ar[r]_\imath & P 
}\end{equation}
where  $P_1=\Z[S_1]$, and $P_2 = \Z[S_2]$.  Equip  the $\ZH$-module $P_1$ with the $\ell_1$-norm $\|\cdot\|_{P_1}$ induced by the free $\Z$-basis $S_1$, and observe that this norm is a filling norm on $P_1$ as a finitely generated $\ZH$-module.  We have that
\[ \|x\|_{P_1} = \| x \|_P   \]
for every $x\in P_1$.  Moreover, if $p=x\oplus y$ is an element of $P \cong P_1\oplus P_2$, we have that
\[  \|x\|_P \leq \|p\|_P.\]

 Observe that we can assume that $\rho \colon P \to Q$ restricts to the trivial morphism when restricted to $P_2$. Simply redefine $\rho$ on $P_2$ as the zero morphism and observe that the commutative diagram of the statement of the proposition still holds.  

  Consider the morphism of finitely generated $\ZH$-modules $\rho\colon P_1 \to Q$.  By   Lemma~\ref{lem:ProjBounded} there is an integer $C_1$ such that for every $x\in P_1$,
\[ \| \rho (x) \|_Q \leq C_1\|x\|_{P_1} .\]

 Let $p\in P$ be arbitrary and suppose $p=x\oplus y$ where $x\in P_1$ and $y\in P_2$. Then   
\[ \|\rho (p)\|_Q = \|\rho (x)\|_Q \leq C_1\|x\|_{P_1} = C_1 \| x\|_P \leq C_1 \| p\|_P ,\]
where the first equality follows from $\rho$ being trivial on $P_2$.

Let $C=\max\{C_0, C_1\}$.  Let $x\in Q$ be an arbitrary element and observe that
\[\|x\|_{Q} = \|\rho\circ \imath (x)\|_Q \leq C\cdot \|\imath(x)\|_P \leq C \cdot \Dist_P^L \left( \|\imath(x)\|_L \right) \leq  C \cdot \Dist_P^L \left( C\cdot \|\imath(x)\|_M \right),\]
where the last inequality follows from the hypothesis that $\|x\|_L\leq C\|x\|_M$ for every $x\in M$.
Therefore for every positive integer $k$, 
\[ \Dist_Q^M(k) \leq C\cdot \Dist_P^L(C  k)\qedhere\]
\end{proof}

\subsection{Proof of Theorem~\ref{thm:subFV}}\label{subsec:mainstep}

\begin{definition} 
The \emph{Bredon cohomological dimension} of $G$ relative to $\mc F$ is at most $n$,  denoted by $\cdF(G) \leq n$, if there is a projective  resolution of $\mc{O_F}(G)$--modules 
\[0\longrightarrow  P_n  \longrightarrow \dots   \longrightarrow P_1 \longrightarrow P_0 \longrightarrow \mathbb \ZGF \rightarrow 0.\] 
\end{definition}
 
\begin{proposition}\label{prop:gammac}
 If  $G$ is of type $F_{\mc F}$  and  $\cdF(G) \leq 2$, then there is an $F_{1, \mc F}$-complex $X$  such that   $Z_1\left(X, \Or\right)$ is finitely generated and free. 
\end{proposition}

The proof of Proposition~\ref{prop:gammac} requires the following definition and lemma.
\begin{definition}
An $\Or$-module $M$ is \emph{stably   free} if there is a finitely generated free $\Or$-module $F$ such that $M \oplus F$ is a  free $\Or$-module.
\end{definition}

\begin{lemma}\label{lem:EilenbergTrick}
Let $X$ be a  $G$-graph and let $v$ be a vertex of $X$ with $G$-stabilizer $K\in \mc F$. Let $X_\star$ be  the $G$-graph obtained from $X$ by adding a new $G$-orbit of edges with representative an edge $e$ such that $G_e=K$ and both endpoints are equal to $v$. Then the $\Or$-module of cellular $1$-cycles $Z_1(X_\star)$ is isomorphic to $Z_1(X) \oplus \Z [\cdot, G/K]$. 
\end{lemma}
\begin{proof}
For each $J\in \mc F$, the component $Z_1 (X_\star)[G/J]$ is identified with $\Z$-module of  cellular $1$-cycles of the fixed point set $X_\star^J$.  Hence  there are natural  isomorphisms   $Z_1 (X_\star)[G/J] \to Z_1(X^J)$ if $J$ is not a subgroup of a conjugate of $K$,  and  $Z_1 (X_\star)[G/J] \to Z_1(X^J)\oplus \Z$ if $J$ is a subgroup of a conjugate of $K$.  One verifies that these isomorphisms of $\Z$-modules  are the components of an isomorphism of $\Or$-modules $Z_1(X_\star) \to Z_1(X) \oplus \Z [\cdot, G/K]$. 
\end{proof}

\begin{proof}[Proof of Proposition~\ref{prop:gammac}]
Let $X$ be a cocompact model for $E_{\mc F}(G)$, and let $n$ be its dimension.
Consider the exact sequences of $\Or$-modules
\begin{equation}\label{eq:corr1} 0\to Z_1(X) \to C_1(X)  \to C_0(X) \to \ZGF \to 0,\end{equation}
and
\begin{equation}\label{eq:corr3} 0\to C_n(X) \to \cdots \to C_1(X)  \to C_0(X) \to \ZGF \to 0.\end{equation}
The assumptions on $X$ imply that  each $C_i(X)$  is a finitely generated and free $\Or$-module. Since $\cdF(G)\leq 2$,    
there is an exact sequence  of projective $\Or$-modules
\begin{equation}\label{eq:corr2} 0\to P_{2} \to  P_1\to P_0 \to \ZGF \to 0.\end{equation}
Applying Schanuel's lemma to the exact sequences~\eqref{eq:corr1} and~\eqref{eq:corr2} implies that $Z_1(X)$ is projective. Then   Schanuel's lemma can be applied to the exact sequences~\eqref{eq:corr1} and~\eqref{eq:corr3},  which yields   that $Z_1(X)$ is stably free and finitely generated.   Let $Y$ be the complex obtained by adding finitely many new orbits of edges to $X$ in such a way that Lemma~\ref{lem:EilenbergTrick} implies that $Z_1(Y)$ is a finitely generated free $\Or$-module.  Then $Y$ is an $F_{1, \mc F}$-space for $G$.
\end{proof}

 \begin{proposition}\label{prop:deltasg}
Let $\Gamma$ be an $F_{1, \mc F}$-complex for $G$. 
If $H$ is a subgroup of $G$ of type $F_{1, \mc F\cap H}$, then $\Gamma$ contains an $H$-equivariant subgraph $\Delta$ which is an $F_{1, \mc F\cap H}$-complex.
\end{proposition}
\begin{proof} 
Since $H$ is $FP_{1,\mc F\cap H}$,  Proposition~\ref{prop:FP1Graph}  implies that there is an $F_{1, \mc F\cap H}$-complex  $\Delta$.    
Let $\{v_i \colon i\in I\}$ and $\{e_j \colon i\in J\}$ be complete sets of representatives of $H$-orbits of vertices and edges of $\Delta$ respectively.  For each $i\in I$, choose a vertex $w_i$ in the fixed point set $\Gamma^{H_{v_i}}$ which is non-empty since $H_{v_i}\in \mc F$. The set assignment $v_i\mapsto w_i$ induces an $H$-map $\phi$ from the vertex set of $\Delta$ to the one of $\Gamma$.  For each $j\in J$, let $H_j$ be the (pointwise) $H$-stabilizer of $e_j$ and let $v_-$ and $v_+$ be the endpoints of $e_j$.  Let $\gamma_j$ be an edge path in the fixed point set $\Gamma^{H_j}$ from $\phi(v_-)$ to $\phi (v_+)$;   $\Gamma^{H_j}$  is connected since $H_j\in \mc F$. After subdividing $\Delta$, the assignment $e_j \mapsto \gamma_j$ extends $\phi$ to an $H$-map of graphs $\phi\colon \Delta \to \Gamma$.  Replacing $\Delta$ with the image of $\phi$ proves the claim.
\end{proof}

\begin{proof}[Proof of Theorem~\ref{thm:subFV}]  
By Proposition~\ref{prop:gammac}, there is an 
$F_{1, \mc F}$-complex $\Gamma$ for $G$ such that $Z_1(\Gamma, \Or)$ is free and  finitely generated.  By Proposition~\ref{prop:deltasg}, $\Gamma$ contains an $H$-equivariant subgraph $\Delta$ which is an $F_{1, \mc F\cap H}$-complex for $H$. Since $H$ is $FP_{2, \mc F\cap H}$, we have that $Z_1(\Delta, \OrH)$ is finitely generated.

From here on, all considered modules are $\OrH$-modules. Consider $\Gamma$ and $\Delta$ as $H$-graphs. Then there is a short exact sequence of cellular $\OrH$--chain complexes 
\begin{equation}\label{eq:main2.9} 0 \rightarrow C_\ast ( \Delta ) \rightarrow C_\ast (\Gamma  ) \rightarrow C_\ast ( \Gamma , \Delta ) \rightarrow 0\end{equation}
where  $C_i( \Gamma , \Delta)$ is the quotient $C_i( \Gamma ) / C_i( \Delta )$.
Consider  the induced long exact homology sequence. Observing that $H_2 (\Gamma , \Delta)$ is trivial, and that  the reduced homology $\tilde H_0 (\Delta)$ is trivial (since $\Delta^K$ is connected for all $K \in \mc F \cap H$), we obtain that the sequence  
\begin{equation}\label{eq:main3} 0 \rightarrow Z_1( \Delta ) \rightarrow Z_1 ( \Gamma  ) \rightarrow Z_1 ( \Gamma , \Delta ) \rightarrow 0\end{equation}
is exact. Considering once more the  long exact sequence induced by~\eqref{eq:main2.9}, one verifies that  
\begin{equation}\label{eq:main5} 0\to Z_1(\Gamma, \Delta) \to C_1(\Gamma, \Delta) \to C_0(\Gamma, \Delta)   \to 0\end{equation}
is an exact sequence of $\OrH$-modules. Since $\Delta$ is an $H$-equivariant subgraph of $\Gamma$, we have that $C_i(\Delta)$ is a free factor of $C_i( \Gamma)$; hence $C_i( \Gamma , \Delta )$ is free. Therefore, the sequence~\eqref{eq:main5} implies that $Z_1(\Gamma, \Delta)$ is   projective. The exact sequence~\eqref{eq:main3} together with   $Z_1(\Gamma, \Delta)$ being projective implies that  
 there is a commutative diagram of $\OrH$-modules
\begin{equation}\nonumber \xymatrix{ 
  & Z_1(\Gamma) \ar[rd]^\rho  &    \\
  Z_1(\Delta) \ar[ur]^\imath \ar[rr]_{Id} &  & Z_1(\Delta).    		
}\end{equation}

The proof concludes by invoking Proposition~\ref{thm:retraction} as follows.  
Considering the $H/1$ component of the above commutative diagram, we obtain the analogous diagram for the $\ZH$-modules of integral cellular $1$-cycles of $\Gamma$ and $\Delta$.  Since $Z_1 (\Gamma, \Or)$ is finitely generated and free, it follows that $P:=Z_1(\Gamma, \Z)$ is a finitely generated and $\mc F$-free $\ZG$-module. Analogously, since $Z_1(\Delta, \OrH)$ is finitely generated, $Q:=Z_1(\Delta, \Z)$ is a finitely generated $\ZH$-module. Moreover, $Q$ is a  submodule of the finitely generated $\ZH$-module $M:=C_1(\Delta, \Z)$, and analogously $P$ is a submodule of the finitely generated $\ZG$-module $L:=C_1(\Gamma, \Z)$. Observe that we can assume that $M\leq L$. Considering the natural $\ell_1$-norms $\|\cdot\|_L$ and $\|\cdot \|_M$ on $L$ and $M$ induced by the collection  $1$-cells (with a chosen orientation) of $\Gamma$ and $\Delta$ respectively, we have 
the equality $\|\sigma\|_L = \|\sigma\|_M$ for every $\sigma\in M$. Then  Remark~\ref{rem:distFV} implies $\FV_{H, \mc F\cap H} = \Dist_Q^M$ and $\FV_{G, \mc F}=\Dist_P^L$. By Proposition~\ref{thm:retraction},  $\FV_{H, \mc F\cap H} \preceq \FV_{G, \mc F}$.
\end{proof}


\section{Relatively Hyperbolic Groups}\label{sec:relhyp}

\subsection{Proof of Theorem~\ref{thm:rhypchar0}(1)}\label{subsec:relhyp}

Let $G$ be a group and let $\mc P$ be a collection of subgroups generating a family $\mc F$. Suppose that $\mc P$ is a finite and almost malnormal collection, that  $G$ is $FP_{2, \mc F}$, and that $\FV_{G, \mc F}(n)$ has linear growth.

By Corollary~\ref{cor:FinEdSt}, there is an $F_{1, \mc F}$-graph $X$ with finite edge stabilizers. By Proposition~\ref{prop:lastlemma}, there is an $F_{1, \mc F}$-space $Y$ with $1$-skeleton $X$, with finitely many $G$-orbits of $2$-dimensional cells,  and such that $\FV_{G, \mc F}$ and $\FV_{Y}$ are equivalent. 
To conclude that $G$ is hyperbolic relative to $\mc F$, it is enough to prove that $X$ is fine and hyperbolic.  Since $\FV_{G, \mc F}$ has linear growth,  the proof concludes by invoking the following result. 

\begin{lemma}\cite[Corollary 3.3]{MP15}\label{cor:suskey}
Let $Y$ be a $1$-acyclic $2$-dimensional cell complex such that there is a bound on the length of  attaching maps of $2$-cells. If $\FV_{Y}(n)$ has linear growth then the $1$-skeleton of $Y$ is a fine hyperbolic graph.\qedhere
\end{lemma}

\subsection{A Class of Relatively Hyperbolic Groups of Bredon Cohomological Dimension $2$}

This last subsection relies on small cancellation theory over free products; our main reference is Lyndon-Schupp texbook on Combinatorial group theory~\cite[Chapter V, Section 9]{LySc01}. The following remark  recalls some well known facts on small cancellation products. 

\begin{remark}\label{prop:malnormal-3}
Suppose that $\mc P=\{P_i\}$ is a finite collection of groups, let $F=\ast P_i$ be the corresponding free product, and let $R$ be a symmetrized subset of $F$ satisfying the $C'(1/6)$-small cancellation condition. Let $N$ be the normal closure of $R$ in $F$.
\begin{enumerate}[leftmargin=*]
\item  The collection $\mc P$ can be regarded as a collection of subgroups of $F/N$. Indeed, Greendlinger's lemma~\cite[Ch. V. Theorem 9.3]{LySc01} implies that $N$ does not contain elements of length less than three. In particular, the natural map $F\to F/N$ embeds each factor $P_i$ of $F$, and no two factors are identified.  
\item  The group $F/N$ is hyperbolic relative to $\mc P$; this is a direct consequence of~\cite[Ch. V. Theorem 9.3]{LySc01} using Osin's approach to relative hyperbolicity~\cite[Page 4, (II)]{Os06}.
\item If each $P\in\mc P$ is torsion-free, then $\mc P$ is malnormal in $F/N$,~\cite[Proposition 2.36]{Os06}.
\end{enumerate}
\end{remark}
 
\begin{theorem}\label{prop:sm}
Suppose that $\mc P=\{P_i\}$ is a finite family of torsion-free groups, that $\mc R$ is a finite symmetrized subset of the free product $F=\ast  P_i$ satisfying the $C'(1/6)$-small cancellation condition and with no proper powers, and that $G=F/N$ where $N$ is the normal closure of $\mc R$. 

If $X$ is the coned-off Cayley complex of a relative presentation $\langle \mc P | \mc  R' \rangle$ where $\mc  R'$ is a minimal subset of $\mc R$ with normal closure $N$ in $F$, then $X$ is a cocompact $E_{\mc F}G$-complex where $\mc F$ is the family generated by $\mc P$. In particular $\gdF(G)\leq 2$.
\end{theorem}
\begin{proof}
In view of Proposition~\ref{prop:F2Fchar} and Remark~\ref{prop:malnormal-3}, it is left to prove that $X$ is contractible.  We argue that every spherical diagram $S\to X$ is reducible as explained below. It follows that $X$ has trivial second homotopy group by a remark of Gersten~\cite[Remark 3.2]{Ge87}. Since $X$ is $2$-dimensional and simply-connected, an application of Hurewicz and Whitehead theorems from algebraic topology imply that $X$ is contractible. 

A spherical diagram is a combinatorial cellular map $S \to X$ where $S$ is topologically the $2$-sphere. Roughly speaking, the diagram is reducible in the sense of Gersten if there are two distinct faces ($2$-cells) of the $2$-sphere $S$  that have an edge in common and map to the same $2$-cell of $X$ by a mirror image, see~\cite[Sec. 3.1]{Ge87} for a precise definition. The small cancellation condition implies that every spherical diagram $S\to X$ is reducible; we sketch an argument below following Ol'shanski\u\i~\cite[Proof of Theorem 13.3]{Ol91}.  

Let $\gamma$ be a path in $X$ whose endpoints are non-cone vertices. We observed in Definition~\ref{def:conedoffCC} that $\gamma$ has as a label a reduced word in the alphabet $\bigsqcup P_i$.  If $\alpha$ is a subpath of $\gamma$ then the label of $\alpha$ is defined as the label of the minimal length subpath of $\gamma$ such that its endpoints are non-cone vertices and it contains $\alpha$ as a subpath. In particular, if $\alpha$ is a trivial path then either $\alpha$ is a non-cone vertex and its label is the empty word, or $\alpha$ is a cone-vertex $v(gP)$ and its label is a non-trivial element of the group $P\in \mc P$. 

Now we remark two direct consequences of Greendlinger's lemma~\cite[Ch. V. Theorem 9.3]{LySc01} together with their interpretations/consequences on the cell structure of $X$. We leave some of the details to the reader.

\begin{itemize}[leftmargin=*]
\item If $r\in \mc R$ has semi-reduced form $r=ab$ and $a$ is not the empty word, then $b\not \in N$. 

 As a consequence, any closed curve $\gamma \to X$ in the $1$-skeleton of $X$ with label $r\in \mc R$ is an embedded circle.
 
\item Let $r,r'\in \mc R$. If $r$ and $r'$ have semi-reduced forms $asbt$ and $aubv$ respectively, and $su^{-1}$ is a reduced form such that $s u^{-1} \in N$, then $r=r'$.  

As a consequence, if $D$ and $D'$ are a pair of $2$-cells of $X$ whose boundaries intersect, then either $D=D'$ or their boundaries intersect in a single path (possibly a trivial path) whose induced label (as a subpath of the boundary of $D$) is a piece. 

This last statement relies on the minimality of $\mc R'$  and the assumption that $\mc R$ has no proper powers, and not only on the stated consequence of  Greendlinger's lemma.  Suppose that there is a pair of distinct $2$-cells $D$ and $D'$ of $X$ such that their boundary intersection contains a path whose label is not a piece. Then from the definition of piece, it follows that the labels $r$ and $r'$ of the boundaries of $D$ and $D'$ respectively are the same reduced word up to taking the inverse and a cyclic conjugation. By minimality of $R'$, the words $r$ and $r'$ correspond to the same element of $R'$, say $r$. By definition of $X$, since $D$ and $D'$ are distinct and have the same boundary path labelled by $r$, the word $r$ is a proper power. But this contradicts that $\mc R$ has no proper powers.

In particular, the consequence of Greendlinger's lemma is only used to show that if the boundary intersection of $D$ and $D'$ is non-empty then it is connected. This is obtained by observing that for a suitable orientations of the boundary paths of $D$ and $D'$, if their intersection is not connected their labels have  have semi-reduced forms $asbt$ and $aubv$ respectively, and such that $su^{-1}$ is the label of an embedded closed path.

\end{itemize}

Suppose that $S\to X$ is a spherical diagram that is not reducible. Let $\Phi$ be the graph whose vertices are the $2$-cells of $S$; and given two  vertices, there is an edge between them for each connected component of the intersection of the boundary paths of the corresponding $2$-cells.  It is immediate that $\Phi$ is planar. Below we argue that every vertex of $\Phi$ has degree at least six, and that $\Phi$ is simplicial, i.e., no double edges and no loops. Since every finite planar  simplicial graph has a vertex of degree at most five, the existence of $\Phi$ yields a contradiction.  Hence a non-reducible spherical diagram $S\to X$ can not exist.

The first bullet statement above implies that $\Phi$ has no edges that connect a vertex to itself.   
Since $S\to X$ is not reducible, the second bullet statement above implies that if $D$ and $D'$ are two distinct $2$-cells of $S$ and their boundaries have non-empty intersection, then they intersect in a single path whose induced label is a piece. This implies that $\Phi$ has no two distinct edges connecting the same vertices. Hence $\Phi$ is simplicial.  Finally, the $C'(1/6)$-condition implies that every $r\in \mc R'$ can not be the concatenation of less than six pieces, and we observed that the intersection of any pair of $2$-cells of $S$ is along a path with label a piece; it follows that every vertex of $\Phi$ has degree at least six.  
\end{proof}

\bibliographystyle{plain}

\end{document}